\documentclass[12pt, reqno]{amsart}
\usepackage{amsmath, amsthm, amscd, amsfonts, amssymb, graphicx, color}
\usepackage[bookmarksnumbered, colorlinks, plainpages]{hyperref}
\hypersetup{colorlinks=true,linkcolor=red, anchorcolor=green, citecolor=cyan, urlcolor=red, filecolor=magenta, pdftoolbar=true}

\usepackage{amsfonts}
\usepackage{color,graphicx,shortvrb}
\usepackage{enumerate}
\usepackage{amsmath}
\usepackage{amssymb}
\usepackage{graphicx}

\usepackage{enumerate}
\usepackage{amsmath}
\usepackage{amssymb}

\usepackage[latin1]{inputenc}

\usepackage{amsmath}

\newtheorem{theorem}{Theorem}[section]

\newtheorem{corollary}[theorem]{Corollary}
\newtheorem{lemma}[theorem]{Lemma}

\newtheorem{problem}[theorem]{Problem}
\newtheorem{proposition}[theorem]{Proposition}

\newtheorem{definition}[theorem]{Definition}
\newtheorem{remark}[theorem]{Remark}
\newtheorem{example}[theorem]{Example}

\usepackage{centernot}

\usepackage{tikz}
\usetikzlibrary{arrows,matrix}

\begin{document}

\title[Linear maps that are $^*$-homomorphisms at a fixed point]{Linear maps between C$^*$-algebras that are $^*$-homomorphisms at a fixed point}

\date{\today}

\author[M.J. Burgos]{Mar{\'i}a J. Burgos}

\address{Departamento de Didáctica de la Matemática, Universidad de Granada,
Facultad de Ciencias de la Educación, Campus Universitario Cartuja s/n, 18011, Granada, Spain}
\email{mariaburgos@ugr.es}

\author[J. Cabello S\'anchez]{Javier Cabello S\'anchez}
\address{Departamento de Matem\'aticas. Facultad de Ciencias de la Universidad de Extremadura.
Avenida de Elvas, s/n, 06011, Spain}
\email{coco@unex.es}

\author[A.M. Peralta]{Antonio M. Peralta}

\address{Departamento de An{\'a}lisis Matem{\'a}tico, Universidad de Granada,
Facultad de Ciencias 18071, Granada, Spain}
\email{aperalta@ugr.es}

\thanks{First and third authors partially supported by the Spanish Ministry of Economy and Competitiveness and European Regional Development Fund project no. MTM2014-58984-P and Junta de Andaluc\'{\i}a grant FQM375.
Second author partially supported by the Spanish Ministry of Economy and Competitiveness project MTM2013-45643-C02-01-P and by the Junta de Extremadura GR15152 IV Plan Regional I+D+i, Ayudas a Grupos de Investigación.
}




\begin{abstract} Let $A$ and $B$ be C$^*$-algebras. A linear map $T:A\to B$ is said to be a $^*$-homomorphism at an element $z\in A$ if $a b^*=z$ in $A$ implies $T (a b^*) =T (a) T (b)^* =T(z)$, and $ c^* d=z$ in $A$ gives $T (c^* d) =T (c)^* T (d) =T(z).$ Assuming that $A$ is unital, we prove that every linear map  $T: A\to B$ which is a $^*$-homomorphism at the unit of $A$ is a Jordan $^*$-homomorphism. If $A$ is simple and infinite, then we establish that a linear map $T: A\to B$ is a $^*$-homomorphism if and only if $T$ is a $^*$-homomorphism at the unit of $A$. For a general unital C$^*$-algebra $A$ and a linear map $T:A\to B$, we prove that $T$ is a $^*$-homomorphism if, and only if, $T$ is a $^*$-homomorphism at $0$ and at $1$. Actually if $p$ is a non-zero projection in $A$, and $T$ is a $^*$-homomorphism at $p$ and at $1-p$, then we prove that $T$ is a Jordan $^*$-homomorphism. We also study bounded linear maps that are $^*$-homomorphisms at a unitary element in $A$.
\end{abstract}

\maketitle

\section{Introduction}

Let $T: A\to B$ be a linear mapping between Banach algebras. A tempting challenge for researchers is to determine conditions on a certain set $\mathcal{S}\subset A\times A$ to guarantee that the property \begin{equation}\label{eq mutliplicative at a subset S} T(a b) = T(a) T(b), \hbox{ for every } (a,b)\in \mathcal{S},
\end{equation} implies that $T$ is a homomorphism. Some particular subsets $\mathcal{S}$ give rise to precise notions studied in the literature. For example, given a fixed element $z\in A$, a linear map $T:A\to B$ satisfying \eqref{eq mutliplicative at a subset S} for the set $\mathcal{S}_{z} = \{ (a,b)\in A\times A : a b =z\}$ is called \emph{multiplicative at $z$}.\smallskip

Let $X$ be a Banach $A$-bimodule. A linear map $S : A\to X$ satisfying \begin{equation}\label{eq derivation} S(ab ) = S(a) b + a S(b),
\end{equation} for every $(a,b)\in \mathcal{S}_{z}$ is called \emph{derivable at $z$}. In \cite[Subsection 4.2]{Bre07}, \cite[Theorem 2]{ChebKeLee}, \cite{Houqi,Jing,JingLuLi2002,ZhuXio2002,ZhuXio2005} the authors study linear maps that are derivable at ${0}$. Linear maps that are derivable at a fixed invertible element $z$ are explored in \cite{Zhu2007}, \cite{ZhuXio2007} and \cite{ZhuXion2008}. A remarkable result due to F. Lu proves that if $z$ is a left or right separating point of $X$ and $T: A\to X$ is a continuous linear map, then $T$ is derivable at $z$ if and only if $T$ is a Jordan derivation and satisfies $T(z a)= T(z) a+ zT(a)$ and $T(a z)=T(a) z + a T(z)$, for all $a\in A$ (see \cite{Lu2009}). For more recent contributions on derivable maps at certain points we refer to \cite{ZhangHouQi2014}, \cite{ZhuXiLi} and \cite{ZhuZhao2013}.\smallskip

Back to multiplicative maps at a certain point, J. Zhu, Ch. Xiong, and H. Zhu explored in \cite{ZhuXiZhu} those linear maps on the algebra $M_n$ of all $n\times n$ matrices that are multiplicative at a certain fixed point $z\in M_n$. More concretely, a fixed point $z\in M_n$ is an \emph{all-multiplicative point} if every linear bijection $\Phi$ on $M_n$  with $\Phi (I_n) = I_n$, which is multiplicative at $z$ is a multiplicative mapping, where $I_n$ denotes the unit matrix in $M_n$. The main result in \cite{ZhuXiZhu} proves that every $z\in M_n$ with det$(z) = 0$ is an all-multiplicative point in $M_n$, and if $\Phi$ is a bijective linear map which is multiplicative at $I_n$, then there exists an invertible matrix $b\in M_n$ such that either $\Phi(x) = bxb^{-1}$ for all $x\in M_n$ or $\Phi(x) = bx^{tr} b^{-1}$ for all $x\in M_n$.\smallskip

Mappings that are multiplicative at the identity element between general unital Banach algebras have been recently considered by J. Li and J. Zhou in \cite{LiZhou11}. One of the main results in the just quoted paper, states that every linear mapping $T$ from a unital Banach $A$ algebra into another Banach algebra $B$ which is multiplicative at $1$ and satisfies $T(a)=T(1)T(a)=T(a)T(1)$, for every $a\in A$, is a Jordan homomorphism (see \cite[Theorem 3.1]{LiZhou11}).\smallskip

Throughout this note, for each Banach algebra $A$ we shall consider its natural Jordan product defined by $a\circ b = \frac12 (a b +ba)$. A Jordan homomorphism between C$^*$-algebras is a linear map preserving the natural Jordan products. A Jordan $^*$-homomorphism between C$^*$-algebras $A$ and $B$ is a Jordan homomorphism $T:A\to B$ satisfying $T(x^*)=T(x)^*$ for every $x\in A$.\smallskip

Other previous results can be now rewritten with the notion of multiplicative map at a point. In \cite[Subsection 4.1]{Bre07}, M. Bre\v{s}ar studies additive maps that are multiplicative at $0$. A forerunner of the previous paper is contained in \cite{ChebKeLee}, where M.A. Chebotar, W.-F. Ke and P.-H. Lee describe the bijective additive maps on a prime ring containing non-trivial idempotents that are multiplicative at zero. The latter three authors together with N.C. Wong \cite{ChebKeLeeWong2003} study linear maps between algebras generated by idempotents that are multiplicative at zero. Their results apply to matrix algebras, standard operator algebras, C$^*$-algebras and von Neumann algebras.\smallskip

J. Araujo and K. Jarosz established in \cite{ArJar03b} that every bijective linear map $T$ between spaces of vector-valued continuous functions such that $T$ and $T^{-1}$ are multiplicative at zero is usually automatically continuous. The same authors prove in \cite[Theorem 1]{ArJar03}  that every linear bijection $T$ between standard operator algebras such that $T$ and $T^{-1}$ are multiplicative at zero is automatically continuous and a multiple of an algebra isomorphism.\smallskip

If $A$ is a unital C$^*$-algebra and $B$ is a Banach algebra, J. Alaminos, M. Bre\v{s}ar, J. Extremera, and A. Villena prove in \cite[Theorem 4.1]{AlBreExtrVill} that a bounded linear operator $h : A \to B$ is multiplicative at zero if and only if $h(1) h(xy) = h(x)h(y)$ for all $x, y \in  A$. Other related results are given in \cite{BurSanchez2013,CuiHou2002,CuiHou2004,LeTsaiWong}.\smallskip

Undoubtedly, C$^*$-algebras enjoy nice geometric properties which make them very attractive when studying results on automatic continuity and on $^*$-homomor-phisms. Motivated by the long list of references commented above, and in order to have a more appropriate notion valid in the particular setting of C$^*$-algebras, we introduce in this note the following definition.

\begin{definition}\label{def *homo at a point} Let $T: A\to B$ be a map between C$^*$-algebras, and let $z$ be an element in $A$. We say that $T$ is a $^*$-homomorphism at $z$ if $a b^*=z$ in $A$ implies $T (a b^*) =T (a) T (b)^*=T(z)$, and $ c^* d=z$ in $A$ gives $T (c^* d) =T (c)^* T (d)=T(z).$
\end{definition}

Clearly, every $^*$-homomorphism between C$^*$-algebras is a $^*$-homomorphism at every point of its domain.\smallskip

Let $A$ and $B$ be C$^*$-algebras, where $A$ is unital. In Theorem \ref{t star-homomorphism at the unit element} we prove that every linear map  $T: A\to B$ which is a $^*$-homomorphism at the unit of $A$ is a Jordan $^*$-homomorphism. If we additionally assume that $A$ is simple and infinite, then a linear map $T: A\to B$ is a $^*$-homomorphism if and only if $T$ is a $^*$-homomorphism at the unit of $A$ (see Theorem \ref{c *hom at 1-simple-infinite}). For a general unital C$^*$-algebra $A$ and a linear map $T:A\to B$, we establish, in Corollary \ref{c *hom at 0 and 1}, that the following conditions are equivalent:
\begin{enumerate}[$(a)$]
\item $T$ is a $^*$-homomorphism,
\item $T$ is a $^*$-homomorphism at $0$ and at $1$.
\end{enumerate} Actually if $p$ is a non-zero projection in $A$, and $T$ is a $^*$-homomorphism at $p$ and at $1-p$, then we prove that $T$ is a Jordan $^*$-homomorphism (see Corollary \ref{c star hom at two complementary projections}).\smallskip

In section \ref{sec:3} we initiate the study of those linear maps $T:A\to B$ that are $^*$-homomorphisms at a fixed unitary element in $A$. Under this hypothesis the conclusions are harder to establish. In Theorem \ref{*-hom at unitary} we prove that if a bounded linear map $T$ from a unital $C^*$-algebra $A$ into another C$^*$-algebra is a $^*$-homomorphism at a unitary $u$ in $A$, then $T(1)$ is a projection, $T(1)T(a)=T(a)T(1)$ for every $a\in A$, and $T(1)T$ is a Jordan homomorphism.  If we also assume that $T(1) T(x) =T(x)$ for every $x\in A$, then $T$ is a Jordan homomorphism. Besides, if $T$  is a $^*$-homomorphism at $0$ and at $u$ then $T(1)$ is a projection, $T(1)$ commutes with $T(A)$, and $T(1)T$ is a $^*$-homomorphism (see Corollary \ref{c *hom at 0 and u}).

\section{$^*$-homomorphism at the unit element}

Multiplicative maps at the identity are connected with Hua's theorem and some of its generalizations. It is well known that every unital Jordan homomorphism between Banach algebras strongly preserves invertibility, that is,
$T(a^{-1})=T(a)^{-1}$, for every invertible element $a\in A$.
Moreover, Hua's theorem shows that every unital
additive map between fields that strongly preserves invertibility is
either an isomorphism or an anti-isomorphism (\cite[Theorem 1.15]{Artin}).\smallskip

In \cite{Chebotar}, M. A. Chebotar, W.-F. Ke, P.-H. Lee, L.-S. Shiao improved Hua's theorem by relaxing the condition $T(a^{-1})=T(a)^{-1}$ to
\begin{equation}\label{Chebo} T(a)T(a^{-1})=T(b)T(b^{-1}),\end{equation} for all $a,b\in A^{-1}$. They prove that every bijective additive map $T:K\to K$ on a division ring $K$,
satisfying (\ref{Chebo}) is of the form $T=T(1)S$, where $S:K\to K$ is an automorphism or an anti-automorphism and $T(1)$ lies in the center of $K$.\smallskip

Later, Y.-F. Lin, T.-L. Wong extended this result to the ring $M_n(K)$ of $n\times n$ ($n\geq 2$) matrices over a division ring $K$ \cite{Lin}. From \cite[Theorem 1.3]{Lin}, if $T:M_n(K)\to M_n(K)$ is a bijective additive map satisfying (\ref{Chebo}), such that $T(1)^2\neq 0$,  then $T(1)$ is a central (invertible) element in $M_n(K)$ and $T(1)^{-1}T$ is an automorphism or an anti-automorphism. Additive maps between general unital Banach satisfying condition (\ref{Chebo}) have been studied by A. M\'arquez-Garc\'ia, A. Morales-Campoy and the first author of this note in \cite{BuMaMo15}.\smallskip

It is clear that every linear map $T:A\to B$ which is multiplicative at the identity verifies (\ref{Chebo}) and $T(1)$ is an idempotent. We also recall that an element $a$ in a Banach algebra $A$ is called Drazin invertible if there exists a (unique) element $b\in A$ such that $ab=ba,$ $bab=b$ and $a^kba=a^k$ for some positive integer $k$. If such $b$ exists, $b$ is called the Drazin inverse of $a$ and is usually denoted by $a^D$.  In particular every idempotent element $p\in A$ is Drazin invertible with $p^D=p$. Taking this into account,
the following result can be deduced directly from \cite[Proposition 2.5]{BuMaMo15}.

\begin{proposition}\label{mult-at-uno} Let $A$ and $B$ be unital Banach algebras, and let $T:A\to B$ be an additive map which is multiplicative at the unit of $A$. Then $T(1)$ commutes with $T(A)$ and $T(1)T$ is a Jordan homomorphism.$\hfill\Box$
\end{proposition}

The proof of the next lemma is left to the reader.

\begin{lemma}\label{l symmetric starhom at z} Let $T : A\to B$ be a linear operator between C$^*$-algebras. Suppose that $T$ is symmetric {\rm(}i.e. $T(a^*) = T(a)^*$, for all $a\in A${\rm)}. Then $T$ is multiplicative at a point $z\in A$ if and only if $T$ is a $^*$-homomorphism at $z$. $\hfill\Box$
\end{lemma}

\begin{lemma}\label{l 1 star-homomorphism at the unit element a} Let $A$ and $B$ be C$^*$-algebras, where $A$ is unital. Suppose $T: A\to B$ is a linear map satisfying that $a b^*=1$ in $A$ implies $T(a) T(b)^* = T(1)$. Then $T(1)$ is a projection. Consequently, the same conclusion holds when $T$ is a $^*$-homomorphism at the unit of $A$.
\end{lemma}

\begin{proof} By assumption $T (1 )= T(1 1^*)= T(1) T(1)^*,$ and hence $T(1)$ is a positive element in $B$. In particular, $T(1) = T(1)^*$ and hence $T(1)^* = T(1) = T(1)^2$, which proves the desired statement.\end{proof}

There is an undoubted advantage in dealing with the geometric properties of C$^*$-algebras. Our next result is a direct consequence of the Russo-Dye theorem \cite{RuDye}, a milestone achievement in the theory of C$^*$-algebras.

\begin{proposition}\label{p automatic continuity of +hom at unit} Let $A$ and $B$ be C$^*$-algebras, where $A$ is unital. Suppose $T: A\to B$ is a linear map
satisfying that $a b^*=1$ in $A$ implies $T(a) T(b)^* = T(1)$. Then $T$ is continuous and $ T(1) T(a) = T(a)$, for every $a\in A$. Furthermore, if $T$ is a $^*$-homomorphism at the unit then $T$ is continuous and $T(a) T(1) = T(1) T(a) = T(a)$, for every $a\in A$.
\end{proposition}

\begin{proof} Let $u$ be a unitary element in $A$. In this case, $$T(1)  = T(u) T(u)^*.$$ We observe that, by Lemma \ref{l 1 star-homomorphism at the unit element a}, $T(1)$ is a projection in $B$. Thus, $T(u)$ is a partial isometry, and $T(u)=T(1) T(u),$ for every unitary $u$ in $A$. The Russo-Dye Theorem implies that $T$ is continuous (see \cite[Corollary 1]{RuDye}). Since every element $a\in A$ with $\|a\|<\frac12$ is a convex combination of four unitaries (compare \cite[page 414]{RuDye}), we also deduce that $T(1) T(a) = T(a)$, for every $a$ in $A$. The final statement follows by similar arguments.
\end{proof}

\begin{theorem}\label{t star-homomorphism at the unit element} Let $A$ and $B$ be C$^*$-algebras, where $A$ is unital. Let $T: A\to B$ be a linear map which is a $^*$-homomorphism at the unit of $A$. Then $T$ is a Jordan $^*$-homomorphism.
\end{theorem}

\begin{proof} Proposition \ref{p automatic continuity of +hom at unit} proves that $T$ is continuous.  Let us take $a=a^*$ in $A$.  Since $e^{ita}$ is a unitary element in $A$ for every $t\in \mathbb{R}$, we deduce that $$ T(1) = T(e^{ita} (e^{ita})^*) = T(e^{ita}) T(e^{ita})^*,$$ for every real $t$. Taking the first derivative in $t$ we get \begin{equation}\label{eq 20 09} 0= i T(a e^{ita}) T(e^{ita})^* -i T(e^{ita}) T(a e^{ita})^*,
 \end{equation}and $$ 0= -i T(a e^{ita})^* T(e^{ita}) +i T(e^{ita})^* T(a e^{ita}),$$ for every real $t$. The case $t=0$ shows that $$  T(a ) T(1)^* = T(1) T(a)^*, \hbox{ for every } a=a^*\in A.$$

Taking a new derivative at $t=0$ in \eqref{eq 20 09} we prove that $$ 0= - T(a^2 ) T(1)^* + T(a ) T(a )^*+ T(a ) T(a )^* - T(1) T(a^2)^*,
$$ for every $a= a^*$. Combining the above identities with the conclusions in Lemma \ref{l 1 star-homomorphism at the unit element a} and Proposition \ref{p automatic continuity of +hom at unit} we have $$T(a^2) = T(a) T(a)^* = T(a)^* T(a) =T(a)^2,$$ for every $a=a^*$.  It is well known that in these circumstances $T$ is a Jordan $^*$-homomorphism.
\end{proof}

\begin{remark}Let $A$ and $B$ be C$^*$-algebras. Assume that $A$ is unital. Let  $S:A\to B$ be a Jordan $^*$-homomorphism. Clearly, $S(1)$ is a projection in $B$ with $S(1) \circ S(a) = S(a\circ 1) = S(a)$ for every $a\in A$. So, replacing $B$ with the C$^*$-subalgebra $C$ of $B$ generated by $S(A)$, we can assume that $B$ is unital and $S(1) =1$.\smallskip

It is known that invertibility can be reformulated in terms of Jordan products, that is, an element $a$ in $A$ is invertible {\rm(}with inverse $a^{-1}${\rm)} if and only if $a\circ a^{-1} =1$ and $a^2 \circ a^{-1} = a$ {\rm(}see \cite[page 51]{Jac}{\rm)}.\smallskip

Take $a,b\in A$ with $ab^*=b^*a=1$, clearly, $a$ is invertible with inverse $a^{-1} = b^*$. Since $S(a)\circ S(b)^* = S(1)$ and $S(a)^2 \circ S(b)^* = S(a)$, we deduce that $S(a)$ in invertible in $C$ with $S(a)^{-1} = S(b)^* = S(b^*)$.  Therefore, $S$ satisfies the following property:
\begin{equation}\label{e more than *hom at 1} ab^*=b^*a=1 \Rightarrow  S(a)S(b)^*=S(b)^*S(a)=S(1).
	\end{equation}

The arguments in the proofs of Lemma \ref{l 1 star-homomorphism at the unit element a}, Proposition \ref{p automatic continuity of +hom at unit} and Theorem \ref{t star-homomorphism at the unit element} are also valid to prove that every linear mapping $S:A\to B$ satisfying \eqref{e more than *hom at 1} is a Jordan $^*$-homomorphism. Therefore, the property in \eqref{e more than *hom at 1} characterizes those linear maps from $A$ into $B$ that are Jordan $^*$-homomorphisms.
\end{remark}

\begin{example}\label{example Pop}{\rm(}\cite[\S 2]{Pop}, \cite[Example 3.13]{Peralta}{\rm)} Let us consider the mapping $T : M_2 (\mathbb{C}) \to M_4 (\mathbb{C})$, defined by $$ T\left( \left(
 \begin{array}{cc}
 a & b \\
 c & d \\
 \end{array}
 \right)
 \right)  = \left(
 \begin{array}{cccc}
 a & 0 & b & 0 \\
 0 & a & 0 & c \\
 c & 0 & d & 0 \\
 0 & b & 0 & d \\
 \end{array}
 \right).$$
It is easy to check that the $T$ is a unital Jordan $^*$-homomorphism. Furthermore, since every one-sided invertible element in  $M_2 (\mathbb{C})$ is invertible, $T$ is a $^*$-homomorphism at the unit element of $M_2 (\mathbb{C})$.\smallskip

Furthermore, the element $\displaystyle a= \left(
\begin{array}{cc}
1 & -1 \\
1 & -1 \\
\end{array}
\right)$ satisfies $a (a^*)^*=(a^*)^* a=0$ and $T(a) T(a) = T(a) T(a^*)^*\neq 0,$  which shows that $T$  is not a $^*$-homomorphism at zero, and hence $T$ is not multiplicative.
\end{example}

In the assumptions of Theorem \ref{t star-homomorphism at the unit element}, obviously, every $^*$-homomorphism $\Phi : A\to B$ is a $^*$-homomorphism at the unit of $A$. Example \ref{example Pop} above shows that the reciprocal is, in general, hopeless. We shall seek for additional hypothesis to conclude that a $^*$-homomorphism at the unit of $A$ is a $^*$-homomorphism.\smallskip

We recall that a unital C$^*$-algebra $A$ is said to be \emph{finite} if its unit is a finite projection. Otherwise, $A$ is called \emph{infinite}. Equivalently, a unital C$^*$-algebra  $A$, is infinite if it contains an infinite projection $p$, that is, $p$ is  Murray-von Neumann equivalent to its subprojections.

\begin{theorem}\label{c *hom at 1-simple-infinite} Let $A$ and $B$ be C$^*$-algebras, where $A$ is simple and infinite. Let $T: A\to B$ be a linear map. Then $T$ is a $^*$-homomorphism if and only if $T$  is a $^*$-homomorphism at $1$.
\end{theorem}

\begin{proof} Assume that $T$ is a $^*$-homomorphism at $1$. By Theorem \ref{t star-homomorphism at the unit element}, $T$ is a Jordan $^*$-homomorphism, $T(1) =p$ is a projection in $B$ and $T(A)\subseteq p B p$. Let $W$ denote the von Neumann subalgebra of $B^{**}$ generated by $T(A)$. Clearly, $W \subseteq p B^{**} p$. Applying \cite[Theorem 3.3]{Stor65}, there exists a central projection $q\in W$ such that $q T : A\to W$, $a\mapsto q T(a)$ is a $^*$-homomorphism and  $(1-q) T : A\to W$, $a\mapsto (1-q) T(a)$ is a $^*$-anti-homomorphism.\smallskip

Since $q$ is central in $W$ and $T$ is a $^*$-homomorphism at $1$, it follows that $q T$ and $(1-q) T$ are $^*$-homomorphisms at $1$.\smallskip

We focus on the map $T_1 = (1-q) T: A\to W$ which is a $^*$-anti-homomorphism. Its kernel, $K= \ker(T_1)$, is a norm-closed $^*$-ideal of $A$. Since $A$ is simple, either $K=\{0\}$ or $K=A$. Suppose that $K=\{0\}$. In this case,  $T_1 : A\to W$ is an injective $^*$-anti-homomorphism, which is a $^*$-homomorphism at $1$.\smallskip

Since $A$ is simple and infinite, we deduce from \cite[Proposition V.2.3.1$(v)$]{Black2006} that $A$ contains a left (or right) invertible element which is not invertible. Pick $a,b\in A$ such that $b^*a=1$ and $ab^*\neq 1$. Applying that $T_1$ is a $^*$-homomorphism at $1$, we get
$$ T_1(1)=T_1(b^*a)=T_1(b)^*T_1(a).$$ Taking into account that $T_1$ is a  $^*$-anti-homomorphism, it follows that
$$T_1(1)=T_1(b)^*T_1(a)=T_1(ab^*),$$ which, by the injectivity of $T_1$, implies that $1=ab^*$, and we get a contradiction. Therefore, $K=A$ and hence $T=qT$ is a $^*$-homomorphism.
\end{proof}

The previous theorem gives conditions on a unital C$^*$-algebra $A$ to guarantee that every $^*$-homomorphism at the unit element is a $^*$-homomorphism. In the next results we explore how to strength the assumptions on the mapping to obtain the same conclusion.

\begin{theorem}\label{cstara-0-1} Let $T: A\to B$ be a linear map between C$^*$-algebras, where $A$ is unital. Suppose that $T$ satisfies the following conditions:
\begin{enumerate}[$(a)$]
\item $ab^*=1\Rightarrow T(a)T(b)^*=T(1)$,
\item $c^*d=0\Rightarrow T(c)^*T(d)=0$.
\end{enumerate} Then $T$ is a $^*$-homomorphism.
\end{theorem}

\begin{proof}
We explore first the consequences derived from condition $(a)$. By Lemma \ref{l 1 star-homomorphism at the unit element a} and Proposition \ref{p automatic continuity of +hom at unit}, we conclude that $T(1)$ is a projection, $T$ is bounded, and $T(a)=T(1)T(a)$, for every $a\in A$. Moreover,
as in the proof of Theorem \ref{t star-homomorphism at the unit element}, given  $a=a^*$ in $A$, since $e^{ita}$ is a unitary element in $A$ for every $t\in \mathbb{R}$, we get $$ T(1) = T(e^{ita} (e^{ita})^*) = T(e^{ita}) T(e^{ita})^*,$$ for every real $t$. Taking the first derivative at $t$ we get $$0= i T(a e^{ita}) T(e^{ita})^* -i T(e^{ita}) T(a e^{ita})^*,$$ for every real $t$. In particular, for $t=0$ we get $T(a)T(1)^*=T(1)T(a)^*$, for every $a=a^*$ in $A$. Therefore $T(x)T(1)=T(1)T(x^*)^*,$ for every $x\in A$.\smallskip

Let us focus on condition $(b)$. Let us define a mapping $\phi:A\times A\to B$ given by $\phi(a,b)=T(a^*)^*T(b),$ for all $a,b \in A$. Clearly, $\phi$ is a bounded bilinear map, and, by condition $(b)$, it satisfies
$$ab=0\Rightarrow \phi(a,b)=0.$$
By \cite[Theorem 2.11]{AlBreExtrVillb}, $\phi(a,bc)=\phi(ab,c)$ for every $a,b,c\in A$. That is,
$$T(a^*)^*T(bc)=T((ab)^*)^*T(c),$$ for every $a,b,c\in A$.
Taking $a=1$ in the above identity, we have
\begin{equation}\label{albre}T(bc)=T(1)^*T(bc)=T(b^*)^*T(c),\end{equation}
for every $b,c\in A$. In particular, given $c=b=b^*$ in $A$,  it follows that $T(b^2)=T(b)^*T(b)$, for every $b=b^*$ in $A$. Therefore, $T$ sends positive elements to positive elements, and hence $T$ is symmetric. Finally, since $T$ is symmetric, identity (\ref{albre}) can be rewritten as $T(bc)=T(b)T(c), \quad(\mbox{for all } b,c\in A),$
which  shows that $T$ is multiplicative, as desired.
\end{proof}

\begin{remark} \label{remcstarb-0-1} Let $T: A\to B$ be a linear map between C$^*$-algebras, where $A$ is unital. We observe that if conditions $(a)$ and $(b)$ in the above Theorem \ref{cstara-0-1} are replaced by
\begin{enumerate}[$(a^{\prime})$]
\item $a^*b=1\Rightarrow T(a)^*T(b)=T(1)$,
\item $cd^*=0\Rightarrow T(c)T(d)^*=0$,
\end{enumerate} we also obtain that $T$ is a $^*$-homomorphism.

\end{remark}

As a direct consequence of Theorem \ref{cstara-0-1} we get the following corollary.

\begin{corollary}\label{c *hom at 0 and 1} Let $T: A\to B$ be a linear map between C$^*$-algebras, where $A$ is unital. The following conditions are equivalent:
\begin{enumerate}[$(a)$]
\item $T$ is a $^*$-homomorphism,
\item $T$  is a $^*$-homomorphism at $0$ and at $1$.$\hfill\Box$
\end{enumerate}
\end{corollary}

It is natural to ask whether the role played by 1 and 0 in the previous corollary can be also played by a non-zero projection and its complement.

\begin{corollary}\label{c star hom at two complementary projections} Let $A$ and $B$ be C$^*$-algebras, where $A$ is unital. Suppose $p$ is a projection in $A$, and let $T: A\to B$ be a linear map which is a $^*$-homomorphism at $p$ and at $1-p$. Then $T$ is a Jordan $^*$-homomorphism.
\end{corollary}

\begin{proof} Let us take $a,b,c,d\in A$ with $a b^* =1$ and $c^* d=1$. By hypothesis we have $$T(p) =  T(c^* d p) = T(c)^* T(dp),$$  and $$T(1- p) =  T(c^* d (1-p)) = T(c)^* T(d(1-p)),$$ which prove $T(1) = T(c)^* T(d)$. Similarly, the identities $T(p) =  T(p a b^*) = T(p a) T(b)^*,$ and  $T(1- p) =  T( (1-p)a b^*) = T((1-p)a) T(b)^*,$ imply that $T(1) = T(a) T(b)^*$. Therefore $T$ is a $^*$-homomorphism at $1$, and hence Theorem \ref{t star-homomorphism at the unit element} concludes that $T$ is a Jordan $^*$-homomorphism.
\end{proof}

\begin{problem}\label{problem  star hom at p and 1-p} Let $T: A\to B$ be a linear map between C$^*$-algebras, where $A$ is unital, and let $p$ be a non-zero projection in $A$. Suppose $T$ is a $^*$-homomorphism at $p$ and at $1-p$. Is $T$ a $^*$-homomorphism?
\end{problem}

\section{$^*$-homomorphisms at some distinguished unitary}\label{sec:3}

In the previous section we conducted a study of linear maps between C$^*$-algebras that are $^*$-homomorphisms at the unit element. Surprisingly, if the unit element is replaced by a unitary element the difficulties are more numerous.

\begin{theorem}\label{*-hom at unitary}  Let $T: A\to B$ be a bounded linear map between C$^*$-algebras, where $A$ is unital. Suppose that $T$ is a $^*$-homomorphism at a unitary $u$ in $A$. Then $T(1)$ is a projection, $T(1)T(a)=T(a)T(1)$, for every $a\in A$, and $T(1)T$ is a Jordan homomorphism.

Furthermore, if we also assume that $T(1) T(x) =T(x)$ for every $x\in A$ {\rm(}that is the case when $T(u)$ is a unitary{\rm)}, then $T$ is a Jordan homomorphism.
\end{theorem}

\begin{proof}
By hypothesis,
\begin{equation}\label{ustar}T(u)=T(1^*u)=T(1)^*T(u), \hbox{ and } T(u)=T(u1^*)=T(u)T(1)^*.\end{equation}
Moreover
\begin{equation}
\label{u-square}T(u)=T(u^*u^2)=T(u)^*T(u^2), \quad T(u)=T(u^2u^*)=T(u^2)T(u)^*.
\end{equation}
Merging these equations we prove
$$T(1)T(u)=T(1)T(u)^*T(u^2)=T(u)^*T(u^2)=T(u),$$
and similarly \begin{equation}\label{1-mult}T(u)T(1)=T(u).\end{equation}

Let $v$ be a unitary element in $A$. By assumptions we have
\begin{equation}\label{unit-1}T(u)=T(v^*(vu))=T(v)^*T(vu),
\end{equation}
\begin{equation}\label{unit-2}T(u)=T((uv)v^*)=T(uv)T(v)^*,
\end{equation}
\begin{equation}\label{unit-3}T(u)=T((vu^*)^*v)=T(vu^*)^*T(v),
\end{equation} and
\begin{equation}\label{unit-4}T(u)=T(v(u^*v)^*)=T(v)T(u^*v)^*,
\end{equation}\smallskip

From identities (\ref{unit-4}) and (\ref{unit-1}) we get
\begin{equation}\label{unit-5}T(u)^*T(vu)=T(u^*v)T(v)^*T(vu)=T(u^*v)T(u).
\end{equation}
Similarly, from identities (\ref{unit-3}) and (\ref{unit-2}) we get
\begin{equation}\label{unit-6}T(uv)T(u)^*=T(uv)T(v)^*T(vu^*)=
T(u)T(vu^*).
\end{equation}
Having in mind again that every element $a\in A$, with $||a||<\frac{1}{2}$ is a convex combination of four unitary elements, we conclude that
\begin{equation}\label{unit-7}
T(u)^*T(xu)=T(u^*x)T(u), \quad T(ux)T(u)^*=T(u)T(xu^*),
\end{equation}or equivalently,
$$T(u)^*T(x)=T(u^*xu^*)T(u), \quad T(x)T(u)^*=T(u)T(u^*xu^*),$$
for every $x\in A$. In particular,
\begin{equation}\label{comm}T(u)T(u)^*T(x)=T(u)T(u^*xu^*)T(u)=T(x)T(u)^*T(u),\end{equation} for every $x\in A$.
Equations \eqref{ustar} and \eqref{comm} show that $T(u)^*T(u)=T(u)T(u)^*$.\smallskip

Let us take $a=a^*$ in $A$ and $t\in \mathbb{R}$. By choosing in (\ref{unit-1}) and (\ref{unit-2}) the unitary  element $v=e^{ita}$  we conclude that $$ T(u) = T(e^{ita})^*T (e^{ita}u) ,\hbox{ and } T(u) = T(ue^{ita}) T(e^{ita})^*,$$ for every real $t$. Taking the first derivative in $t$ we get
 \begin{equation}\label{unit-eq first derivative a} 0= -i T(a e^{ita})^* T(e^{ita}u) +i T(e^{ita})^* T(a e^{ita}u),
\end{equation} and \begin{equation}\label{unit-eq first derivative b} 0= i T(ua e^{ita})T(e^{ita})^* -i T(ue^{ita})T(a e^{ita})^*,
\end{equation} for every real $t$.
Replacing $t$ by $0$ in the above identities, we obtain
$$T(a)^*T(u)=T(1)^*T(au), \quad T(u)T(a)^*=T(ua)T(1)^*,$$ for every $a=a^*$ in $A$. By linearizing we obtain:
\begin{equation}\label{unit-8}T(x^*)^*T(u)=T(1)^*T(xu), \quad T(u)T(x^*)^*=T(ux)T(1)^*,\end{equation} for every $x$ in $A$.
For $x=u^*$, the above equalities give
\begin{equation}\label{proj}T(u)^*T(u)=T(1)^*T(1), \quad T(u)T(u)^*=T(1)T(1)^*.\end{equation}

From identitities  (\ref{ustar}), (\ref{1-mult}) and (\ref{proj}) it follows that
$$T(u)=T(u)T(1)=T(u)T(1)^*T(1)=T(u)T(u)^*T(u).$$
This shows that $T(u) $ (and hence $T(1)$) is a partial isometry. Furthermore
$$T(1)^*=T(1)^*T(1)T(1)^*=T(u)^*T(u)T(1)^*=T(u)^*T(u)=T(1)^*T(1),$$ which shows that $T(1)$ is  in fact a projection.\smallskip

Now, from (\ref{comm}) and (\ref{proj}) we get \begin{equation}\label{eq Ta and Tx commute} T(1)T(x)=T(u)T(u)^*T(x)=T(x)T(u)^*T(u)=T(x)T(1),
\end{equation} for all $x\in A$.\smallskip

Taking a new derivative at $t=0$ in (\ref{unit-eq first derivative a}) and (\ref{unit-eq first derivative b}) we obtain that \begin{equation}\label{unit-eq second derivative a} 0= - T(a^2 )^* T(u) + T(a ) ^ *T(au )+ T(a )^* T(au ) - T(1) ^*T(a^2u),
\end{equation} and \begin{equation}\label{unit-eq second derivative b} 0= - T(ua^2 )T(1)^*+ T(ua) T(a ) ^*+  T(ua)T(a)^* - T(u)T(a^2)^*,
\end{equation} for every $a= a^*$ in $A$.
Having into account (\ref{unit-8}) in the last two equalities we deduce that
\begin{equation}
\label{unit-jordan a} T(a^2)^*T(u)=T(1)^*T(a^2u)=T(a)^*T(au),
\end{equation}and
\begin{equation}
\label{unit-jordan b} T(u)T(a^2)^*=T(u a^2)T(1)^*=T(ua)T(a)^*,
\end{equation} for all $a=a^*$ in $A$.\smallskip

Now, since $T(1)$ is a projection, we know from \eqref{eq Ta and Tx commute} and the first equality in \eqref{unit-8} (with $x= a u^*$) that $$T(a)T(1)=T(1)T(a)=T(ua)^*T(u) = T(1)^* T(a)= T(1) T(a),$$ and thus by \eqref{unit-jordan b}
$$T(a^2)T(1)=T(a^2)T(u)^*T(u)=T(a)T(ua)^*T(u)= T(a) T(1) T(a) =T(a)^2T(1), $$ for every $a=a^*$ in $A$. We have therefore proved that $T(1)$ is a projection commuting with $T(A)$ and that $T(1)T$ is a Jordan homomorphism.\smallskip

For the final statement, we observe that if $T(u)$ is a unitary, then, by \eqref{proj}, $T(1)$ also is a unitary in $B$ and hence $T(1)=1$. The rest of the statement is clear.
\end{proof}

We can obtain a better conclusion by strengthening the hypothesis in the previous theorem.

\begin{corollary}\label{c *hom at 0 and u} Let $T: A\to B$ be a bounded linear map between C$^*$-algebras, where $A$ is unital, and let $u$ be a unitary element in $A$. If $T$  is a $^*$-homomorphism at $0$ and at $u$ then $T(1)$ is a projection, $T(1)$ commutes with $T(A)$, and $T(1)T$ is a $^*$-homomorphism. If we also asume that $T(1) T(x) = T(x)$, for every $x\in A$ then $T$ is a $^*$-homomorphism.
\end{corollary}

\begin{proof}
We know from Theorem \ref{*-hom at unitary} that $T(1)$ is a projection, $T(1)T(a)=T(a)T(1)$, for every $a\in A$, and that $T(1)T$ is a Jordan homomorphism. Moreover, from identities (\ref{unit-7})
\begin{equation}\label{unit-7-1} T(u)^*T(xu)=T(u^*x)T(u),
\end{equation}
for every $x\in A$.
Arguing as in Theorem \ref{cstara-0-1}, the map $\phi:A\times A\to B$ defined by $\phi(a,b)=T(a^*)^*T(b),$ for all $a,b \in A$ is a bilinear bounded map and  satisfies
$$ab=0\Rightarrow \phi(a,b)=0.$$
By \cite[Theorem 2.11]{AlBreExtrVillb}, $\phi(a,bc)=\phi(ab,c)$ for every $a,b,c\in A$. That is,
\begin{equation}\label{starO} T(a^*)^*T(bc)=T((ab)^*)^*T(c),
\end{equation} for every $a,b,c\in A$. In particular, given $x \in A$, it follows from the last identity that
\begin{equation}\label{unit-7-2}T(u)^*T(xu)=T((u^*x)^*)^*T(u)=T(x^*u)^*T(u).\end{equation}
Merging identities (\ref{unit-7-1}) and (\ref{unit-7-2}) we obtain
$$T(u^*x)T(u)=T(x^*u)^*T(u),$$
or equivalently, multiplying by $T(u)^*$ on the right  and replacing $x$ by $ux$, the identity $$T(x)T(1)=T(x^*)^*T(1),$$ holds for every $x\in A$.
This shows that $T(1)T$ is a $^*$-homomorphism.
Finally, identity (\ref{starO}) ensures that
$T(1)T(bc)=T(b)T(c),$ for every $b,c$ in $A$. Since $T(1)$ is a projection commuting with $T(A)$, it follows that $T(1)T$ is multiplicative, as desired. The rest is clear.
\end{proof}


\end{document}